\documentclass[12pt,a4paper]{article}
\usepackage{amsmath,amsfonts,amssymb,amsthm}
\usepackage{graphicx}
\usepackage[T1]{fontenc}

\usepackage{todonotes}
\newcommand{\e}{\varepsilon}
\newcommand{\dd}{\mathrm{d}}

\renewcommand{\div}{{\rm{div}}}

\newcommand{\C}{\mathcal{C}}

\newcommand{\R}{\mathbb{R}}

\newcommand{\Omeps}{\Omega^{\varepsilon}}

\newcommand{\psia}{\psi_{\alpha}}

\newcommand{\Omeet}{\Omega^{\frac{\varepsilon}{2}}}

\newtheorem{theorem}{Theorem}[section]
\newtheorem{lemma}[theorem]{Lemma}
\newtheorem{remark}[theorem]{Remark}

\newtheorem{corollary}[theorem]{Corollary}
\newtheorem{prop}[theorem]{Proposition}
\numberwithin{equation}{section}

\begin{document}
\title{Uniqueness for degenerate parabolic equations in weighted $L^1$ spaces }
\author{Camilla Nobili\thanks{Departement Mathematik, Universit\"at Hamburg, Germany (camilla.nobili@uni-hamburg.de).}\,\, and Fabio Punzo\thanks{Dipartimento di Matematica, Politecnico di Milano, Italia (fabio.punzo@polimi.it).}}


\date{}

\maketitle
  \begin{abstract}
We study uniqueness of solutions to degenerate parabolic problems, posed in bounded domains, where no boundary conditions are imposed. Under suitable assumptions on the operator, uniqueness is obtained for solutions that satisfy an appropriate integral condition; in particular, such condition holds for possibly unbounded solutions belonging to a suitable weighted $L^1$ space. 
  \end{abstract}

\bigskip

\noindent{\it 2010 Mathematics Subject Classification: 35K10, 35K15, 35K65}

\noindent {\bf Keywords:} Degenerate parabolic equations; uniqueness of solutions; weighted Lebesgue spaces.

\section{Introduction}
We investigate uniqueness of solutions to degenerate parabolic problems of the following type
\begin{equation}\label{TD}
 \begin{cases}
  \partial_tu=\div\{a(x,t)\nabla u\}+f & \mbox{ on } \Omega\times (0,T]=:Q_T\\
  u=u_0& \mbox{ on } \Omega \times \{0\}\,,
 \end{cases}
\end{equation}
where $\Omega\subset \R^n$ is an open bounded subset and $T>0$. Note that in \eqref{TD} no boundary conditions have been prescribed.
Concerning the coefficient
$a(x,t)$ and the data $f$ and $u_0$, we always assume that $$a\in C_{x,t}^{1,0}(Q_T), a\geq 0, a\not\equiv 0 \mbox{ in } Q_T\,,$$
$f\in C(Q_T), u_0\in C(\Omega)$. Furthermore, we assume that $\partial \Omega$ is a manifold of dimension $n-1$ of class $C^3$.

\smallskip

A wide literature is devoted to degenerate elliptic and parabolic problems, based on both  analytical methods (see e.g. \cite{Fich}, \cite{FP1}-\cite{FP4}, \cite{MPP}-\cite{PuTe1}, \cite{Taira}) and stochastic calculus (see, e.g.,  \cite{Khas}, \cite{SV}). Under appropriate assumptions on the behaviour at the boundary of the coefficients of the operator,  in \cite{Fich} it is shown that uniqueness of solutions can hold without prescribing boundary conditions at some portion of the boundary. Such solutions belong to $C^2(Q_T)\cap C(\bar Q_T)$, therefore they are bounded.

In \cite{PPT}, \cite{PoPT}, by means of appropriate super-- and subsolutions, similar uniqueness results have been obtained,  also for unbounded solutions.  It is assumed that the solutions satisfy suitable pointwise growth conditions near the boundary. Such conditions are related to the constructed super-- and subsolutions.

In \cite{Pu} uniqueness in the weighted Lebesgue space $L^1_{d^{\beta}(x)}(Q_T)\, (\beta>0)$ is shown for degenerate operators in non-divergence form, under appropriate conditions on the coefficients, similar to those in \cite{Fich}. Here and hereafter,
$$d(x):={\rm dist}(x,\partial\Omega)\quad (x\in \Omega)\,$$
is the function {\it distance from the boundary}.

In \cite{Pu2}, under suitable hypotheses on the coefficient $a$, uniqueness results for problem \eqref{TD}, in suitable weighted $L^2$ spaces, are established, by developing a general idea used, for instance, in \cite{Grig2} and in \cite[Theorem 9.2]{Grig} (see also \cite{IM}) for different purposes. Such uniqueness results are obtained as a consequence of suitable integral maximum principles. Note that integral maximum principles in the whole $\mathbb R^n$  for solutions of degenerate parabolic equations are also obtained in \cite{AB}, \cite{EKT}.

\smallskip

In this paper we generalize the uniqueness results in \cite{Pu2}, since we enlarge the uniqueness class. In fact, we now consider solutions belonging to an appropriate weighted $L^1$ space. The passage from $L^2$ to $L^1$ causes important changes in the proofs. Let us outline the differences between our methods and results, and those in \cite{Pu2}.
The line of arguments in \cite{Pu2} is the following: multiply the differential equation in \eqref{TD} by suitable test functions, integrate by parts {\it one time} and obtain convenient estimates on the solution. To do this, an important step is to find a function $\xi(x,t)$, depending on the distance function $d(x)$, which is Lipschitz continuous w.r.t. to $x$ and $C^1$ w.r.t. to $t$, and satisfies
\begin{equation}\label{eq1af}
\partial_t \xi(x,t) + \alpha\, a(x,t) |\nabla \xi(x,t)|^2\leq 0 \quad \textrm{for a.e. } x\in \Omega, \; \textrm{for any } t\in (0, \bar T),
\end{equation}
for appropriate $\alpha>0, \bar T>0.$

Now, suppose that, for some $ \gamma>1, c_0>\tilde c_0>0, c_1>0,$ for all
$(x,t)\in Q_T$,
\begin{equation}\label{a3}
\tilde c_0 d^{\gamma}(x) \leq a(x,t)\leq c_0 d^{\gamma}(x) \\ \quad \mbox{ and } \\ \quad |\nabla a(x,t)|\leq c_1 d^{\gamma-1}(x)\,.
\end{equation}
For every $\varepsilon>0$ let
\[\Omega^{\varepsilon}:=\{x\in \Omega\,:\, d(x)>\varepsilon\}.\]
In the present paper to obtain uniqueness in a weighted $L^1$ space, we argue as follows: we multiply the differential equation in \eqref{TD} by suitable test functions, then we integrate by parts {\it two times}. Hence to get convenient bounds on the solution, we have to control new terms that appear after the second integration by part. A crucial point in the proof is to exhibit a function $\xi=\xi(x,t)$ with $\xi(\cdot, t)\in C^2(\Omega\setminus\partial\Omeps)\cap C^1(\Omega), \xi(x, \cdot)\in C^1(\Omega)$, which satisfies
\begin{equation}\label{eq2af}
\partial_t\xi+{\rm{div}}\big\{a(x,t)\nabla \xi\big\}+\frac 52 a(x,t)|\nabla\xi|^2 \leq 0  \quad \textrm{in }   [\Omega\setminus \partial\Omeps]\times(T_1, T_2),
\end{equation}
and
\begin{equation}\label{eq3af}
\frac{\partial \xi}{\partial n_{\varepsilon}}=0\quad \text{in }\,\, \partial\Omeps\times(T_1, T_2)\,,
\end{equation}
where $n_\varepsilon$ is the unit outward normal vector to $\Omega^\varepsilon$ at $\partial\Omeps$, for appropriate $0<T_1<T_2.$
Observe that $\xi(x,t)$ is defined in terms of the distance function from the boundary and its behaviour as $x\to\partial\Omega$ is very important, since it influences the integral conditions for the solutions, which guarantees uniqueness.
Clearly, the construction of $\xi$ fulfilling \eqref{eq2af} and \eqref{eq3af} is more delicate than that verifying only \eqref{eq1af}. The choice of $\xi$ changes according to whether  $\gamma>2$ or $\gamma\in [1, 2]$; consequently, in these two cases the proofs present some important differences.

\smallskip

The paper is organized as follows. In Section \ref{result} we state our main two uniqueness results, concerning the two cases $\gamma>2$ and $\gamma\in [1, 2]$; in addition, we compare them with some related results in the literature. The uniqueness result for $\gamma>2$ is proved in Section \ref{sth1}, while the other one, for $\gamma\in [1, 2]$, in Section \ref{sth2}.

\section{Statements of the results }\label{result}

%
%

%

%
%

Consider the homogeneous problem associated to \eqref{TD}, that is
\begin{equation}\label{TD2}
 \begin{cases}
  \partial_tu=\div\{a(x,t)\nabla u\} & \mbox{ in } Q_T\\
  u=0& \mbox{ on } \Omega \times \{0\}\,.
 \end{cases}
\end{equation}

\smallskip

The following two uniqueness results are our main contribute in this paper.

\begin{theorem}\label{th1}
 Suppose that $u\in C^{2,1}(Q_T)\cap C(\Omega\times [0,T])$ solves \eqref{TD2} and $a$ satisfies \eqref{a3} with $\gamma >2$.
 Moreover, suppose that, for some $C>0, \theta>0, \varepsilon_0>0$,
\begin{equation}\label{eq4af}
  \int_0^T\int_{\Omeps}|u(x,t)|\, dx\, dt\leq Ce^{\theta\, \varepsilon^{-\gamma+2}} \qquad \mbox{ for every } \varepsilon\in(0,\varepsilon_0).
\end{equation}
Then $u\equiv 0$ in $Q_T$.

\end{theorem}

Obviously, there exist {\it unbounded} functions satisfying condition \eqref{eq4af}. For any $\phi\in C(\Omega), \phi>0$, $p\geq 1$, let
\begin{equation*}
L^p_{\phi}(Q_T):=\left\{u: Q_T \to \mathbb R \textrm{ measurable} : \int_0^T\int_{\Omega} |u(x,t)|^p \phi(x)\, dx dt<\infty \right\}\,.
\end{equation*}
It is direct to see that if  $u\in L^1_{\phi}(Q_T)$ with $\phi(x)=e^{\{-\theta[d(x)]^{2-\gamma}\}}, \theta>0, \gamma>2$, then condition \eqref{eq4af} holds.

\begin{theorem}\label{th2}
  Suppose that $u\in C^{2,1}(Q_T)\cap C(\Omega\times [0,T])$ solves \eqref{TD2} and $a$ satisfies \eqref{a3} with $\gamma \in [1,2]$.
  Moreover, suppose that, for some $C>0, \varepsilon_0>0$ and $\mu>-2\gamma+4$,
  \begin{equation}\label{eq5af}
   \int_0^T\int_{\Omega^{\frac{\varepsilon}2}\setminus\Omega^{\frac 23 \varepsilon}} |u(x,t)|\, [d(x)]^{\gamma-2}\, dx\, dt\leq C\varepsilon^{\mu} \qquad \mbox{ for every } \varepsilon\in(0,\varepsilon_0).
  \end{equation}
  Then $u\equiv 0$ in $Q_T$.
\end{theorem}

Note that, if $u\in L^1_{\phi}(Q_T)$ with
$\phi(x)=[d(x)]^{\gamma-2-\mu}$, then \eqref{eq5af} is valid.

Furthermore, if $\gamma\in \left(\frac 5 3, 2\right]$ and
\begin{equation}\label{eq9af}
|u(x,t)|\leq  \bar C[d(x)]^{-l}\quad \textrm{for every}\;\; x\in \Omega, t\in [0, T]\,,
\end{equation}
for some $\bar C>0$ and $0< l<3\gamma-5$, then \eqref{eq5af} holds with $\mu=\gamma-1-l>-2\gamma+4.$

\begin{remark}{\em
(i) Theorem \ref{th1} generalizes \cite[Theorem 2.1]{Pu2}, where  \eqref{eq4af} is replaced by the stronger condition
\begin{equation}
  \int_0^T\int_{\Omeps}|u(x,t)|^2\, dx\, dt\leq Ce^{\theta\, \varepsilon^{-\gamma+2}} \qquad \mbox{ for every } \varepsilon\in(0,\varepsilon_0).
\end{equation}

\noindent (ii) Theorem \ref{th2} generalizes \cite[Theorem 2.2]{Pu2}, where  \eqref{eq5af} is replaced by the stronger condition
 \begin{equation}
   \int_0^T\int_{\Omega^{\frac{\varepsilon}2}\setminus\Omega^{\frac 23 \varepsilon}} |u(x,t)|^2\, [d(x)]^{\gamma-2}\, dx\, dt\leq C\varepsilon^{\mu} \qquad \mbox{ for every } \varepsilon\in(0,\varepsilon_0).
\end{equation}
for some $\mu>0$. However, note that in Theorem \ref{th2} the further request $\mu>-2\gamma+4$ is made.

\noindent (iii) We should note that in \cite[Theorem 2.1, 22]{Pu2} the hypothesis on the coefficient $a$ is weaker. In fact, instead of \eqref{a3} it is only assumed that
\[a(x,t)\leq c_0 d^\gamma(x)\quad \textrm{for all }\,\, (x, t)\in Q_T\,.\]}
 \end{remark}

\smallskip

 \begin{remark}{\em
 Let $\gamma\in (1, 2]$ and $u$ be a solution of problem \eqref{TD2} satisfying \eqref{eq9af}, for some $\bar C>0$ and $l>0.$
 Observe that \cite[Theorem 2.2]{Pu2} yields that if $0<l<\frac{\gamma-1}2$, then $u\equiv 0$ in $Q_T$.

 Now, let $\gamma \in \left(\frac 5 3,2 \right].$ From Theorem \ref{th2} and the subsequent comments, it follows that $u\equiv 0$, provided that $0<l<3\gamma-5$. Since
 $$3\gamma-5\in (0,1),$$
 while
 $$\frac{\gamma-1}2\in \left(0,\frac 12\right],$$
  the growth condition for $u$ in Theorem \ref{th2} is weaker than that in \cite[Theorem 2.2]{Pu2}. On the other hand, when $\gamma\in \left(0, \frac 53\right)$, \cite[Theorem 2.2]{Pu2} can be applied, whereas the hypotheses of Theorem \ref{th2} are not verified (under the extra condition \eqref{eq9af}).

  Finally, recall that in view of \cite[Proposition 3.3]{Pu2}, if $\gamma=1, l=0$, then uniqueness holds in $L^\infty(Q_T)$. }
  \end{remark}

\medskip

By Theorems \ref{th1} and \ref{th2} the following uniqueness result immediately follows.
\begin{corollary}\label{cor1}
Let $u_1, u_2\in C^{2,1}(Q_T)\cap C(\Omega\times [0, T])$ be two solutions of problem \eqref{TD}.
Assume that \eqref{a3} holds with $\gamma>2$  and both $u_1$ and $u_2$ satisfy condition \eqref{eq4af}, or that \eqref{a3} holds with $\gamma\in [1, 2]$ and both $u_1$ and $u_2$ satisfy condition \eqref{eq5af}. Then $u_1\equiv u_2$ in $Q_T$.
\end{corollary}

\begin{remark}\label{oss1}{\em  Assume that, for some $\varepsilon>0, C_0>0$ and $\gamma\in \mathbb R$,
\begin{equation}\label{e99}
a(x,t)=C_0 [d(x)]^{\gamma} \quad \textrm{for any}\;\; x\in \Omega\setminus\Omega^\varepsilon, t\in [0, T]\,.
\end{equation}

If $\gamma\geq 2$, the results in \cite{Fich} give uniqueness of solutions to problem \eqref{TD2} in $C^2(Q_T)\cap C(\bar Q_T)$. So, in  particular such solutions are bounded. Hence, our results are in agreement with those in \cite{Fich} in the special case of {\em bounded solutions}, if $\gamma>2$. Instead, when $\gamma<2$, the results in \cite{Fich} cannot be applied, since the coefficients are not regular enough.

The results in \cite{PoPT} could be applied, once we construct suitable super-- and subsolutions; however, we would obtain uniqueness under pointwise growth conditions near $\partial \Omega$. Finally, the results in \cite{PPT} and in \cite{Pu} cannot be applied, since our operator does not satisfy the required hypotheses. }
\end{remark}

From the existence result in \cite[Proposition 3.1]{Pu2} and Corollary \ref{cor1} we get the following existence and uniqueness result.

\begin{corollary}\label{cor2}
Let $f\equiv 0, \gamma>2$ and $a>0$ in $Q_T$. Suppose that, for some $0<\beta\leq \gamma-2, \tau>0,$
\begin{equation}\label{eq6af}
0\leq u_0 \leq \exp\left\{\frac{[d(x)]^{\beta}}{\tau}\right\}\quad \textrm{for all}\;\; x\in \Omega\,.
\end{equation}
Assume that \eqref{a3} holds. Then there exists a solution $u\in C^{2,1}(Q_T)\cap C(\Omega\times [0, T])$ of problem \eqref{TD} fulfilling
 \begin{equation}\label{eq7af}
0\leq u(x,t) \leq \hat C\exp\left\{\frac{[d(x)]^{\beta}}{\tau-\lambda t}\right\}\quad \textrm{for all}\;\; x\in \Omega\,, t\in \left[0, T\right],
\end{equation}
with $T=\frac{\tau}{2\lambda}$, for suitable $\lambda>0, \hat C>0$.
Furthermore, $u$ is the unique solution of problem \eqref{TD} in  $L^1_{\phi}(Q_T)$ with $\phi(x)=e^{\{-\frac2{\tau}[d(x)]^{2-\gamma}\}}$.

\end{corollary}

Observe that Theorems \ref{th1} and \ref{th2} imply uniqueness whenever \eqref{a3} holds with $\gamma\geq 1$. Such request on $\gamma$ is indeed optimal. In fact, from \cite[Proposition 3.2]{Pu2} it follows that when, for some $\varepsilon>0, \gamma<1, c_2>0, c_3>0, s \in [0, \gamma)$,
\begin{equation}\label{eq8af}
c_2 [d(x)]^{\gamma}\leq a(x)\leq c_3 [d(x)]^{\gamma-s}\quad \textrm{for all}\;\; x\in \Omega\setminus\Omega^{\varepsilon}\,,
\end{equation}
problem \eqref{TD} admits {\it infinitely} many bounded solutions.

\begin{remark}{\em We observe that there are important differences between
problem \eqref{TD} and the companion problem
\begin{equation}\label{e1aa}
\left\{
\begin{array}{ll}
 \,  \partial_t u = a(x,t)\Delta u + f  \, &\textrm{in}\,\,Q_T
\\&\\
\textrm{ }u \, = u_0& \textrm{in}\,\,  \Omega\times \{0\} \,.
\end{array}
\right.
\end{equation}
For example, let
\[a(x,t)=[d(x)]^{\gamma}\quad (\gamma>1)\,.\]
If $\gamma\geq 2$, then there exists a unique bounded solutions to problem \eqref{e1aa} (see \cite[Section 7]{IM}, \cite[Theorem 2.16]{PPT}). On the other hand, if $\gamma<2$, then nonuniqueness of solutions of problem \eqref{e1aa} prevails, in the sense that it is possible to prescribe Dirichlet boundary data at $\partial \Omega\times (0, T]$ (see \cite[Section 7]{IM}, \cite[Theorem 2.18]{PPT}).
Thus, the change between uniqueness and nonuniqueness occurs for $\gamma=2$. Instead, such change for problem \eqref{TD} occurs for $\gamma=1$.}
\end{remark}
\section{Proof of Theorem \ref{th1}}\label{sth1}
Observe that
\begin{equation}\label{grad-d}
|\nabla d(x)|\leq 1\quad \text{for a.e. }\,\, x\in \Omega\,.
\end{equation}
Moreover, (see e.g. \cite{PPT}) if $\partial \Omega$ is of class $C^3$, then there exists $\varepsilon_0\in (0,1)$ such that for each $\varepsilon\in (0, \varepsilon_0)$  $d\in C^2(\Omega\setminus\Omega^\varepsilon),$
and, for some $k_0>0$,
\begin{equation}\label{laplace-d}
|\Delta d(x)|\leq k_0\quad \text{in }\,\,\Omega\setminus\Omega^\varepsilon\,.
\end{equation}
In addition, there exists $\nu_0\in (0,1)$ such that
\begin{equation}\label{lb-gradd}
|\nabla d(x)|\geq \nu_0\quad \text{ for any }\,\, x\in \Omega\setminus\Omega^\varepsilon\,.
\end{equation}

\medskip

For each $\beta>0$, define the function
\begin{equation}\label{zeta}
 \zeta(x,t):=\begin{cases}
           0 & \mbox{ if } x\in \Omega^{\varepsilon}\\
           [d(x)]^{-\beta}-\varepsilon^{-\beta} &\mbox{ if } x\in \Omega\setminus\Omega^{\varepsilon}\,
           \end{cases}\,.
\end{equation}
Differentiating the function above we have
\begin{equation}\label{nabla-zeta}
\nabla \zeta(x,t)=-\beta[d(x)]^{-\beta-1}\nabla d(x)\qquad \mbox{ for any } x\in \Omega\setminus \Omega^{\varepsilon}\,,
\end{equation}
thus
$$|\nabla \zeta(x,t)|^2\leq \beta^2[d(x)]^{-2\beta-2}\qquad \mbox{ for any } x\in \Omega\setminus \Omega^{\varepsilon}\,.$$
Finally define the function
\begin{equation}\label{xi}
 \xi(x,t):=-\frac{\zeta^2(x)}{2(s-\alpha_1t)}
\end{equation}
for any $x\in \Omega, t\neq \frac{s}{\alpha_1}$, where here $\alpha_1>0$ is a parameter to be chosen later.
Note that $\xi(\cdot, t)\in C^2(\Omega\setminus \partial\Omeps)\cap C^1(\Omega)$ and
\begin{equation}\label{14bis}
  \frac{\partial\xi(x,t)}{\partial n_{\varepsilon}}=0 \qquad \mbox{ for any } x\in \partial\Omeps, \; t\neq \frac{s}{\alpha_1}\,,
\end{equation}
where $n_{\varepsilon}$ is the outward normal to $\Omeps$.

Let $\gamma>2$, $c\in \left(0,\frac12\right)$ be such that
\begin{equation}\label{eqagg21}
[(1-c)^{-\frac{\gamma-2}{2}}-1](c_1+c_0k_0)-\beta\nu_0\tilde c_0 <0\,,
\end{equation}
and define
\begin{equation}\label{eqagg22b}
\sigma:=1-(1-c)^{\frac{\gamma-2}{2}}\,.
\end{equation}

\smallskip

The proof of Theorem \ref{th1} is based on the combination of the following results.
\begin{prop}\label{lem1}
  Under assumption \eqref{a3} with $\gamma>2$, suppose $u\in C^{2,1}(Q_T)\cap C(\Omega\times [0,T])$ solves \eqref{TD2}.
  Suppose that, for some $C>0$ and $\theta>0$, \eqref{eq4af} holds.
  Let $\tau\in (0, T), c\in \left(0,\frac 12\right)$ be such that \eqref{eqagg21} is satisfied, $\sigma$ be defined by \eqref{eqagg22b},
  \begin{equation*}
  0<\delta<\min\left\{
 \frac{\sigma^2}{(\gamma-2)(c_1+c_0)},\; \tau,\; \frac{\left[\left(\frac 32 \right)^{\frac{\gamma-2}{2}}-1\right]^2}{4\theta\alpha_1}\right\}\,
\end{equation*}
and
\begin{equation*}
\alpha_1\geq\max\left\{\frac{10 c_0 (\gamma-2)^2}{\sigma^2},\; \frac54 c_0(\gamma-2)^2 \right\}\,.
\end{equation*}

  %

 %
 %

Then
\begin{equation}
\begin{array}{rlll}
\int_{\Omeps}|u(x,\tau)|\, dx
&\leq&\int_{\Omeet}|u(x,\tau-\delta)|\, dx + \tilde{C}\varepsilon^{\gamma-2}\,,
\end{array}
\end{equation}
where $\tilde{C}>0$ is a suitable constant independent of $\varepsilon$.
\end{prop}
\begin{lemma}\label{lemma2}
Let $u\in C(\Omega\times [0, T])$ with
\begin{equation}\label{e8}
u=0\quad \textrm{in}\;\; \Omega\times\{0\}\,.
\end{equation}
Suppose that there exist $c>0, \bar \e>0, \mu>0, \hat{C}>0$ such that for any $\e\in (0, \bar \e)$, $\tau\in (0,T)$ and
\begin{equation}\label{e2}
0<\delta\leq \min\{\tau, c\}\,,
\end{equation}
there holds
\begin{equation}\label{e3}
\int_{\Omega^{\e}} |u(x, \tau)|\,dx \leq \int_{\Omega^{\frac{\epsilon}2}} |u(x, \tau-\delta)|\, dx + \hat{C} \e^{\mu}\,.
\end{equation}
Then
\[u\equiv 0\quad \textrm{in}\;\; \Omega\times (0, T]\,.\]
\end{lemma}

\medskip

Now we are ready to prove Theorem \ref{th1}.
\begin{proof}[Proof of Theorem \ref{th1}] We obtain the thesis, combining Proposition \ref{lem1} and Lemma \ref{lemma2} with $$\mu=\gamma-2,$$ $$c=\min\left\{
\frac{\sigma^2}{(\gamma-2)(c_1+c_0)},\; \frac{\left[\left(\frac 32 \right)^{\frac{\gamma-2}{2}}-1\right]^2}{4\theta\alpha_1}\right\},$$
and $$\tilde{C}=\hat{C}\,.$$
\end{proof}
\subsection{Proofs of Proposition \ref{lem1} and Lemma \ref{lemma2}}

Consider a family of cut-off functions $\{\eta_{\varepsilon}\}\subset C^{\infty}(\Omega)$ such that
$$0\leq \eta_{\varepsilon}\leq 1\,,$$
and
\begin{equation}\label{eta-0}
 \eta_{\varepsilon}=\begin{cases}
                    1 & \mbox{ in } \Omega^{\frac 23\varepsilon}\\
                    0 & \mbox{ in } \Omega \setminus \Omega^{\frac{\varepsilon}{2}}\,.
                   \end{cases}
\end{equation}
Notice that
\begin{equation}\label{eta-1}
\begin{array}{rlll}
 |\nabla \eta_{\varepsilon}|&\leq& \frac{A_1}{\varepsilon} \quad \mbox{ for every } x\in \Omega\,, \\
 \\
  |\Delta \eta_{\varepsilon}|&\leq& \frac{A_2}{\varepsilon^2} \quad \mbox{ for every } x\in \Omega\,,
\end{array}
\end{equation}
where $A_1$ and $A_2$ are two positive constants.

For every $\alpha>0$, consider a function $\psia:\R\rightarrow \R^{+}$ of class $C^2$ such that
\begin{equation}\label{def:psi''}
\psia''\geq 0 \,.
\end{equation}
Then, by the chain rule
\begin{equation*}\label{sub-solution}
  \div\{a(x,t)\nabla \psia(u)\}=\psia'(u)\div\{a(x,t)\nabla u\}+\psia''(u)a(x,t) |\nabla u|^2
\end{equation*}
and, because of \eqref{def:psi''} and the positivity of $a$, we can estimate the second term on the right-hand side from below, obtaining
\begin{equation}\label{eqagg40}
    \div\{a(x,t)\nabla \psia(u)\}\geq \psia'(u)\div\{a(x,t)\nabla u\}=\partial_t \psia(u)\,,
\end{equation}
where in the last identity we used equation \eqref{TD2}.
%
%
Thus the composed function $\psia(u)$ is a subsolution of \eqref{TD}.
\smallskip

The main ingredient for the proof of Proposition \ref{lem1} is the following

\begin{lemma}\label{pr1}
 Under assumption \eqref{a3} with $\gamma>2$, suppose $u\in C^{2,1}(Q_T)\cap C(\Omega\times [0,T])$ solves \eqref{TD2}.
 Let $0<\varepsilon<\varepsilon_0, \tau\in (0, T), c\in \left(0,\frac 12\right)$ be such that \eqref{eqagg21} is satisfied, $\sigma$ be defined by \eqref{eqagg22b}. If
 \begin{equation}\label{eqagg2}
 0<\delta<\min\left\{
\frac{\sigma^2}{(\gamma-2)(c_1+c_0)}, \tau\right\}
\end{equation}
 and
 \begin{equation}\label{eqagg3}
\alpha_1\geq\max\left\{\frac{10 c_0 (\gamma-2)^2}{\sigma^2}, \frac54 c_0(\gamma-2)^2\right\}\,,
 \end{equation}
 %
 %
 then
\begin{equation}\label{bound-prop1}
\begin{array}{rlll}
  \int_{\Omeet}\psia(u(x,\tau))\eta^2(x)e^{\xi(x,\tau)}\, dx
  &\leq&\int_{\Omeet}\psia(u(x,\tau-\delta))\eta^2(x)e^{\xi(x,\tau-\delta)}\, dx &\\
  \\
  &+& C_1\varepsilon^{ \gamma-2 } \iint_{\Omeet\setminus \Omega^{\frac 23 \varepsilon} \times (\tau-\delta, \tau)}\psia(u(x,t)) e^{\xi(x,t)} \, dx\, dt\,,
\end{array}
\end{equation}
where $\xi$ is defined in \eqref{xi} with $s=\alpha_1(\tau+\delta)$ and $C_1>0$ is a suitable constant independent of $\varepsilon$.
\end{lemma}

\begin{proof}[Proof of Lemma \ref{pr1}]
Define the set $\C=\Omega^{\frac{\varepsilon}{2}}\times(\tau-\delta,\tau)$. Testing the time derivative of $\psia$ with $\eta^2e^{\xi(x,t)}$ we get
\begin{eqnarray}\label{0-put-here}
\int_{\Omeet}\psia(u(x,\tau))\eta(x)^2e^{\xi(x,\tau)}\, dx
&=&\int_{\Omeet}\psia(u(x,\tau-\delta))\eta^2(x)e^{\xi(x,\tau-\delta)}\, dx\notag\\
&&+\iint_{\C} \partial_t[\psia(u(x,t))]\eta^2(x)e^{\xi(x,t)}\, dx\, dt\notag\\
&&+ \iint_{\C} \psia(u(x,t))\eta^2(x)\partial_te^{\xi(x,t)}\, dx\, dt\,.\notag\\
\end{eqnarray}
We can compute the second term of the right-hand-side of the above equation as
\begin{eqnarray}\label{1-put-here}
 \iint_{\C}\partial_t[\psia(u)]\eta^2e^{\xi}\, dx\, dt\notag
 &=&\iint_{\C}\psia'(u)\partial_tu\,\eta^2e^{\xi}\, dx\, dt\notag\\
 &=&\iint \psia'(u)\div\{a\nabla u\}\eta^2e^{\xi}\, dx\, dt\notag\\
&\stackrel{\eqref{sub-solution}}{=}&\iint_{\C} \left(\div\{a\nabla \psia(u)\}-\psia'' a |\nabla u|^2\right)\eta^2e^{\xi}\, dx\, dt\notag\\
&\leq&\iint_{\C} \div\{a\nabla \psia(u)\}\eta^2e^{\xi}\, dx\,dt\,,\notag\\
\end{eqnarray}
where, in the last inequality we used the positivity of the integrating factor, due to \eqref{def:psi''} and \eqref{a3}.
%
%
We need to estimate the right-hand side of the above inequality further: Integrating by parts a second time we have
%
\begin{eqnarray*}
 &&\iint_{\C}\div\{a\,\nabla \psia(u)\}\eta^2e^{\xi}\, dx\,dt\\
 &&=
 -\iint_{\C}a\,\nabla \psia(u)\cdot \nabla (\eta^2e^{\xi})\, dx\, dt\\
 &&=-\iint_{\C}a\,\nabla \psia(u)\cdot 2\eta \nabla\eta\, e^{\xi}\, dx\, dt
 - \iint_{\C}a\,\nabla \psia(u)\cdot \eta^2 \nabla\xi\, e^{\xi}\, dx\, dt\,.
\end{eqnarray*}
Observe that, since $\xi\in C^2(\Omega\setminus \partial\Omeps)$, the last identity is justified by splitting the integral in theset $\Omeet\setminus \Omeps$ and $\Omeps$, integrating by parts and eliminating the boundary terms thanks to \eqref{14bis} and \eqref{eta-0}:

\begin{eqnarray*}
  &&-\iint_{\C}a\,\nabla \psia(u)\cdot 2\eta \nabla\eta\, e^{\xi}\, dx\, dt
  - \iint_{\C}a\,\nabla \psia(u)\cdot \eta^2 \nabla\xi\, e^{\xi}\, dx\, dt\\
&&=-\iint_{\Omeps}a\,\nabla \psia(u)\cdot 2\eta \nabla\eta\, e^{\xi}\, dx\, dt
- \iint_{\Omeps\times (\tau-\delta,\tau)}a\,\nabla \psia(u)\cdot \eta^2 \nabla\xi \,e^{\xi}\, dx\, dt\\
&&  \qquad-\iint_{\Omega^{\frac{\varepsilon}{2}}\setminus \Omega^{\varepsilon}\times (\tau-\delta,\tau)}a\,\nabla \psia(u)\cdot 2\eta \nabla\eta \,e^{\xi}\, dx\, dt
- \iint_{\Omega^{\frac{\varepsilon}{2}}\setminus \Omega^{\varepsilon}}a\,\nabla \psia(u)\cdot \eta^2 \nabla\xi \,e^{\xi}\, dx\, dt\\
&&=\iint_{\Omeps\times (\tau-\delta,\tau)}\div(a\,  2\eta \nabla\eta \,e^{\xi}) \psia(u)\, dx\, dt
-\int_{\partial\Omeps}a\,  2\eta \frac{\partial\eta}{\partial n_{\varepsilon}} e^{\xi} \psia(u) \, dS_x\, dt\\
&& \qquad+ \iint_{\Omeps\times (\tau-\delta,\tau)}\div(a\, \eta^2 \nabla\xi \,e^{\xi})\psia(u)\, dx\, dt
- \int_{\partial\Omeps}a\, \eta^2 \frac{\partial\xi}{\partial n_{\varepsilon}} e^{\xi}\psia(u)\, dS_x\, dt\\
&&\qquad+\iint_{\Omega^{\frac{\varepsilon}{2}}\setminus \Omeps\times (\tau-\delta,\tau)}\div(a\,  2\eta \nabla\eta \,e^{\xi}) \psia(u)\, dx\, dt\\
&&\qquad-\int_{\partial\Omega^{\frac{\varepsilon}{2}}\cup \partial\Omeps \times (\tau-\delta,\tau) }a\,  2\eta \frac{\partial\eta}{\partial n_{\varepsilon}} e^{\xi} \psia(u) \, dS_x\, dt\\
&&\qquad+ \iint_{\Omega^{\frac{\varepsilon}{2}}\setminus \Omeps \times (\tau-\delta,\tau)}\div(a\, \eta^2 \nabla\xi \,e^{\xi})\psia(u)\, dx\, dt \\
&&\qquad-\int_{\partial\Omega^{\frac{\varepsilon}{2}}\cup \partial\Omeps \times (\tau-\delta,\tau)}a\, \eta^2 \frac{\partial\xi}{\partial n_{\varepsilon}} e^{\xi}\psia(u)\, dS_x\, dt\\
&&\stackrel{\eqref{14bis}\&\eqref{eta-0} }{=}\iint_{\Omeps\times (\tau-\delta,\tau)}\div(a\,  2\eta \nabla\eta \,e^{\xi}) \psia(u)\, dx\, dt\\
&&\qquad+ \iint_{\Omeps\times (\tau-\delta,\tau)}\div(a\, \eta^2 \nabla\xi \,e^{\xi})\psia(u)\, dx\, dt\\
&&\qquad+\iint_{\Omega^{\frac{\varepsilon}{2}}\setminus \Omeps\times (\tau-\delta,\tau)}\div(a\,  2\eta \nabla\eta \,e^{\xi}) \psia(u)\, dx\,dt\\
&&\qquad+ \iint_{\Omega^{\frac{\varepsilon}{2}}\setminus \Omeps\times (\tau-\delta,\tau)}\div(a\, \eta^2 \nabla\xi \,e^{\xi})\psia(u)\, dx\, dt\,.
\end{eqnarray*}

We therefore obtained

\begin{eqnarray*}
 &&\iint_{\C}\div\{a\,\nabla \psia(u)\}\eta^2e^{\xi}\, dx\,dt\\
&&= -2\iint_{\C}\{\eta \,e^{\xi}{\rm{div}} (a\,\nabla \eta)+|\nabla \eta|^2a\,e^{\xi}+\eta e^{\xi}a\,\nabla\xi\nabla \eta\}\psia(u)\, dx\, dt\\
&&\qquad-\iint_{\C}\{2\eta\nabla \eta \,e^{\xi}a\,\nabla \xi+\eta^2 e^{\xi}|\nabla \xi|^2 a\,+\eta^2 \,e^{\xi}{\rm{div}} (a\,\nabla \xi)\}\psia(u)\, dx\, dt\,.
\end{eqnarray*}
and, inserting this new expression in \eqref{1-put-here} and this last one back into inequality \eqref{0-put-here}, we get
\begin{equation*}
\begin{array}{rlll}
&&\int_{\Omeet}\psia(u(x,\tau))\eta^2(x)e^{\xi(x,\tau)}\, dx\\
\\
&&\leq\int_{\Omeet}\psia(u(x,\tau-\delta))\eta^2(x)e^{\xi(x,\tau-\delta)}\, dx &\\
\\
&& \qquad-2\iint_{\C}\left\{\eta e^{\xi}{\rm{div}} (a(x,t)\nabla \eta(x))+|\nabla \eta(x)|^2a(x,t)e^{\xi(x,t)}\right.\\
\\
&&\qquad\qquad\left.+\eta(x) e^{\xi(x,t)}a(x,t)\nabla\xi(x,t)\nabla \eta(x)+2\eta(x)\nabla \eta(x) e^{\xi(x,t)}a(x,t)\nabla \xi(x,t)\right.\\
\\
&&\qquad\qquad\left.+\eta^2(x) e^{\xi(x,t)}|\nabla \xi(x,t)|^2 a(x,t)+\eta^2(x) e^{\xi(x,t)}{\rm{div}} (a(x,t)\nabla \xi(x,t))\right\}\\
\\
&&\hspace{11cm} \times\psia(u(x,t))\, dx\, dt \quad \\
\\
&&\qquad+ \iint_{\C} \psia(u(x,t))\eta^2(x)\partial_te^{\xi(x,t)}\, dx\, dt  \,.
\end{array}
\end{equation*}
Using Young's inequality in the form
 \[
 3\iint_{\C}\psia \eta\nabla\eta \nabla \xi a\, e^{\xi}\leq\frac 32\iint_{\C}\psia|\nabla \xi|^2\eta^2 a\,e^{\xi}\, dx\, dt+\frac 32 \iint_{\C}\psia |\nabla \eta|^2a\,e^{\xi}\, dx\, dt\,
\]
in the second integral of the right-hand side, we get
\begin{eqnarray*}
 &&-2\iint_{\C}\{\eta e^{\xi}{\rm{div}} (a\,\nabla \eta)+|\nabla \eta|^2a\,e^{\xi}+\eta e^{\xi}a\,\nabla\xi\nabla \eta\\
\\
&&\qquad +2\eta\nabla \eta e^{\xi}a\,\nabla \xi+\eta^2 e^{\xi}|\nabla \xi|^2 a\,+\eta^2 e^{\xi}{\rm{div}} (a\,\nabla \xi)\}\psia(u)\, dx\, dt \\
\\
&&\leq\iint_{\C}\psia(u) e^{\xi} [\eta{\rm{div}} (a\,\nabla\eta)+|\nabla \eta|^2a]\, dx\, dt\\
\\
&&\qquad+\iint_{\C}\psia(u) e^{\xi}[\eta^2|\nabla \xi|^2a+\eta^2{\rm{div}} (a\,\nabla \xi)]\, dx\, dt\\
&&\qquad+\frac 32\iint_{\C}\psia(u)\,e^{\xi}|\nabla \xi|^2\eta^2a\,dx\, dt+\frac 32\iint_{\C}\psia(u)\,e^{\xi}|\nabla \eta|^2a\, dx\, dt.
\end{eqnarray*}

Putting all the previous estimates together we have
\begin{equation*}
\begin{array}{rlll}
&&\int_{\Omeet}\psia(u(x,\tau))\eta^2(x)e^{\xi(x,\tau)}\, dx\\
\\
&&\quad\leq\int_{\Omeet}\psia(u(x,\tau-\delta))\eta^2(x)e^{\xi(x,\tau-\delta)}\, dx &\\
\\
&&\quad+\iint_{\C}\psia(u(x,t)) e^{\xi(x,t)} [\eta{\rm{div}} (a(x,t)\nabla\eta(x))+|\nabla \eta(x)|^2a(x,t)]\, dx\, dt\\
\\
&&\quad+\int_{\C}\psia(u(x,t)) e^{\xi(x,t)}[\eta^2(x)|\nabla \xi(x,t)|^2a(x,t)+\eta^2{\rm{div}} (a(x,t)\nabla \xi(x,t))]\, dx\, dt\\
\\
&&\quad+\frac 32\iint_{\C}\psia(u(x,t))|\nabla \xi(x,t)|^2\eta^2(x)a(x,t)e^{\xi(x,t)}\, dx\, dt\\
\\
&&\quad+\frac 32\iint_{\C}\psia(u(x,t))|\nabla \eta(x)|^2a(x,t)e^{\xi(x,t)}\, dx\, dt\\
\\
&&\quad+ \iint \psia(u(x,t))\eta^2(x)e^{\xi(x,t)}\partial_t\xi(x,t) \, dx\, dt .\\
\\
\end{array}
\end{equation*}
Finally, summing up and rearranging the terms we have
\begin{equation}\label{eq:4}
\begin{array}{rlll}
&&\int_{\Omeet}\psia(u(x,\tau))\eta^2(x)e^{\xi(x,\tau)}\, dx\\
\\
&&\leq\int_{\Omeet}\psia(u(x,\tau-\delta))\eta^2(x)e^{\xi(x,\tau-\delta)}\, dx &\\
\\
&&\quad+ \iint_{\C}\psia(u(x,t)) e^{\xi(x,t)} [\eta{\rm{div}} (a(x,t)\nabla\eta(x))+\frac 52|\nabla \eta(x)|^2a(x,t)]\, dx\, dt\\
\\
&&\quad+\iint_{\C}\psia(u) e^{\xi(x,t)}\eta^2[\partial_t\xi(x,t)+\frac 52|\nabla \xi(x,t)|^2a(x,t)+{\rm{div}} (a(x,t)\nabla \xi(x,t))]\, dx\, dt\,.
\end{array}
\end{equation}
Our next goal is to show that
\begin{align*}
&\partial_t\xi+\frac 52 a\,|\nabla\xi|^2+{\rm{div}}(a\,\nabla \xi)\leq 0 \qquad \mbox{ in } [\Omega\setminus\partial\Omeps]\times (\tau-\delta,\tau)\,;\tag{E1}\\
\\
 &\eta\, {\rm{div}} (a\,\nabla\eta)+\frac 52|\nabla \eta|^2a\leq C_1\varepsilon^{\gamma-2} \qquad \mbox{ in }  \Omega\times (\tau-\delta,\tau)\,, \tag{E2}
 \end{align*}
for some $C_1>0$ independent of $\varepsilon.$

%
%
\smallskip

{\textbf{Claim 1}}: Condition (E1) holds.

{\it Proof of Claim 1.}
 We start recalling that, by definition of $\zeta$, the function $\xi$ (and all its derivatives in time and space) are supported in $\Omega\setminus\Omeps=\{x\in \Omega\;|\; d(x)\leq \varepsilon\}$ so (E1) is trivially verified in $\Omeps$. Now, consider any $x\in \Omega\setminus\overline{\Omega^\varepsilon}$  and any  $t\in (\tau-\delta, \tau).$
In view of the definition of $\xi$ \eqref{xi} we compute
\begin{equation}
 \partial_t\xi =-\frac{\alpha_1\zeta^2}{2(s-\alpha_1t)^2},\qquad |\nabla\xi|^2=\frac{\zeta^2|\nabla \zeta|^2}{(s-\alpha_1t)^2}\\
\end{equation}
and
\begin{equation}\label{last-term}
{\rm{div}}(a(x,t)\nabla \xi)=-\nabla a\cdot \frac{\zeta\nabla \zeta}{s-\alpha_1t} -a(x,t)\frac{|\nabla\zeta|^2}{s-\alpha_1 t}-a(x,t)\frac{\zeta\Delta\zeta}{s-\alpha_1 t}\,.
\end{equation}
%
%
%
Rewriting the right-hand side of \eqref{last-term} by inserting the definition of $\nabla \zeta$ and $\Delta \zeta$, we have
\begin{equation}
  \begin{array}{rlll}
    \nabla a\cdot\zeta\nabla \zeta&=&-\beta\,\zeta\, \nabla a\cdot \nabla d \,d^{-\beta-1} \\
    \\
    a\zeta\Delta \zeta&=&\beta(\beta+1)\,a\,\zeta\, d^{-\beta-2}\,|\nabla d|^2-\beta \,a\,\zeta\, d^{-\beta-1}\,\Delta d\\
    \\
    a|\nabla \zeta|^2&=&\beta^2\,a\, d^{-2\beta-2}\,|\nabla d|^2\,.
  \end{array}
\end{equation}
Putting all previous terms together we obtain the expression
%
%
%
\begin{equation}\label{estimate:1}
\begin{array}{rlll}
 &&\partial_t\xi+\frac 52 a(x,t)|\nabla\xi|^2+{\rm{div}}(a\nabla \xi)\\
 \\
 &&=\frac{1}{2(s-\alpha_1 t)^2}\left\{- \alpha_1\zeta^2+5 a\zeta^2\beta^2\,d^{-2(\beta+1)}\,|\nabla d|^2\right.\\
 \\
 &&\qquad+\left.2(s-\alpha_1t)\beta\left[\zeta\,\nabla a\cdot\nabla d  d^{-\beta-1}
 +a\, \zeta\, \Delta d \,d^{-\beta-1}\right.\right.\\
 \\
 &&\qquad\left.\left.-\beta \,a \,d^{-2\beta-2}\,|\nabla d|^2
 -(\beta+1)\, a\, \zeta\,d^{-\beta-2}\,|\nabla d|^2
 \right]\right\}\,.
 \end{array}
 \end{equation}
 %
%
In order to estimate the right-hand side of the expression above, we decompose the set $\Omega\setminus\overline{\Omega^{\varepsilon}}$ as
$$\Omega\setminus\overline{\Omega^{\varepsilon}}=(\Omega\setminus \Omega^{\varepsilon-\varepsilon_1})\cup (\Omega^{\varepsilon-\varepsilon_1}\setminus \overline{\Omega^{\varepsilon}})\,.$$
\begin{figure}
  \centering
 \includegraphics[scale=0.4]{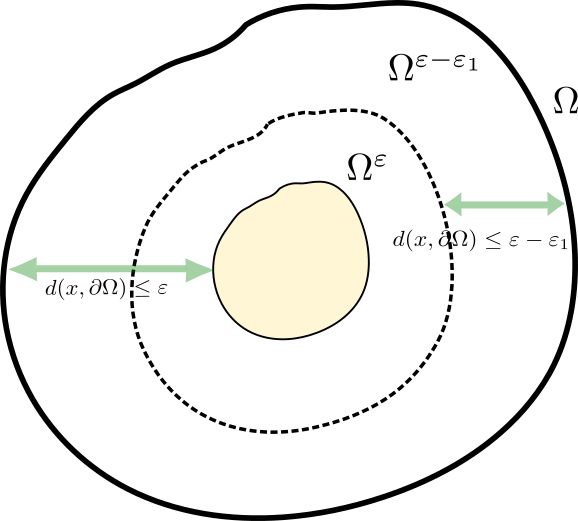}
\caption{Illustation of the decomposition of the set $\Omega$.}
\end{figure}
Consider first the region
 $$\Omega\setminus \Omega^{\varepsilon-\varepsilon_1}=\{x\in \Omega\; |\; d(x)\leq \varepsilon-\varepsilon_1\} \quad \mbox{ with } \quad \varepsilon_1\in (0,\varepsilon/2)\,.$$
 %

 Thanks to \eqref{lb-gradd}, the last two terms of the right-hand side can be bounded from above by zero, i.e.
 $$-\beta\,a\,d^{-2\beta-2}\,|\nabla d|^2\leq 0$$
 and
 $$-(\beta+1) a\,d^{-\beta-2}\,|\nabla d|^2 \zeta\leq 0\,.$$
 Using \eqref{laplace-d} and \eqref{a3}
 and $\zeta(x)=d^{-\beta}-\varepsilon^{-\beta}\leq d^{-\beta}$ we can estimate
 $$a\Delta d\,d^{-\beta-1}\,\zeta\leq c_0 k_0d^{-2\beta-1+\gamma}\,,$$
 while using \eqref{grad-d} and \eqref{a3}
 we have
 $$\nabla a\cdot \nabla d\,d^{-\beta-1}\,\zeta\leq c_1 d^{-2\beta-2+\gamma}\,.$$
 Finally
 $$5a\,\zeta^2\,\beta^2\,d^{-2(\beta+1)}\,|\nabla d|^2\leq10 \beta^2 c_0\,d^{-4\beta-2+\gamma}\,,$$
 where we used
 \begin{equation}\label{ub1}
 \zeta^2=d^{-2\beta}+\varepsilon^{-2\beta}-2d^{-\beta}\varepsilon^{-\beta}\leq 2d^{-2\beta}\,.
 \end{equation}
 We claim that there exists a $\sigma\in (0,1)$ such that
 \begin{equation}\label{ub2}
 -\zeta^2(x)\leq - \sigma^2 d^{-2\beta}(x)\,,
 \end{equation}
 and this will follow if we can show
 $$d^{-\beta}(x)-\varepsilon^{-\beta}(x)\geq \sigma d^{-\beta}(x)\,.$$
 The letter is equivalent to
 $$\sigma \leq1-\left(\frac{d}{\varepsilon}\right)^{\beta}\,,$$
 which is clearly fulfilled by choosing
 $$\sigma=1-\left(\frac{\varepsilon-\varepsilon_1}{\varepsilon}\right)^{\beta}\,.$$
 Let $\varepsilon_1=c\varepsilon$ with $c\in \left(0,\frac 12\right)$. Hence
 \begin{equation*}\label{eqagg1}
 \sigma=1-(1-c)^{\beta}.
 \end{equation*}
  Now we can use \eqref{ub2} and \eqref{ub1} to estimate the right-hand side of \eqref{estimate:1} further:
 \begin{eqnarray*}
  &&\partial_t\xi+\frac 52 a\,|\nabla\xi|^2+{\rm{div}}(a\,\nabla \xi)\\
  &&\leq\frac{1}{2(s-\alpha_1 t)^2}
  \Big\{-\alpha_1 \sigma^2 d^{-2\beta} +10 \beta^2\,c_0\,d^{-4\beta-2+\gamma}\\ &&\qquad+ (s-\alpha_1t)\beta\left[c_1 d^{-2\beta-2+\gamma}+c_0d^{-2\beta-1+\gamma}
  \right]\Big\}\\
  \\
  &&\leq \frac{d^{-4\beta-2+\gamma}}{2(s-\alpha_1 t)^2}\left\{-\alpha_1\sigma^2 d^{2\beta+2-\gamma}+10\beta^2c_0+(s-\alpha_1t)\beta\left[c_1 d^{2\beta}+c_0d^{2\beta+1}\right]\right\}\,.
  \end{eqnarray*}
Choose
\begin{equation}\label{condition-beta}
\beta=\frac{\gamma-2}{2}\,,
\end{equation}
\begin{equation}\label{eqagg35}
s=\alpha_1(\tau +\delta)\,.
\end{equation}
So, for all $t\in (\tau-\delta, \tau),$
\begin{equation}\label{eqagg20}
\alpha_1\delta<s-\alpha_1t<2\alpha_1\delta\,.
\end{equation}
%
This together with the fact that
$$d(x)\leq\varepsilon \quad \mbox{ in } \Omega \setminus \Omega^{\varepsilon}$$
yields
\begin{equation}\label{estimate:2}
\begin{array}{rlll}
  &&\partial_t\xi+\frac 52 a\,|\nabla\xi|^2+{\rm{div}}(a\,\nabla \xi)\\
  \\
  &&\leq\frac{d^{-4\beta+\gamma-2}}{2(s-\alpha_1 t)^2}\left\{-\alpha_1\sigma^2+10\beta^2c_0+ 2\alpha_1 \delta\beta\left[c_1 \varepsilon^{2\beta}+c_0\varepsilon^{2\beta+1}\right]\right\}\,.
  \end{array}
\end{equation}

If we impose that
\begin{equation}\label{condition-2}
 0<\delta\leq\frac{\sigma^2}{ 2 \beta\left[c_1 \varepsilon^{2\beta}+c_0\varepsilon^{2\beta+1}\right]}\,,
\end{equation}
then from \eqref{estimate:2} we get
\begin{eqnarray*}
  \partial_t\xi+\frac 52 a|\nabla\xi|^2+{\rm{div}}(a\nabla \xi)
  &\leq&\frac{d^{-4\beta-2+\gamma}}{2(s-\alpha_1 t)^2}\left\{-\frac{\alpha_1}{2}\sigma^2+10c_0\beta^2\right\}\,.
  \end{eqnarray*}
Now, observe that in view of assumption \eqref{eqagg2}, condition \eqref{condition-2} is true.
Finally, if
\begin{equation}\label{condition-3}
 \alpha_1\geq\frac{40 c_0 \beta^2}{\sigma^2}\,,
\end{equation}
then
\begin{eqnarray*}
  \partial_t\xi+\frac 52 a\,|\nabla\xi|^2+{\rm{div}}(a\,\nabla \xi)
  &\leq&\frac{d^{-4\beta-2+\gamma}}{2(s-\alpha_1 t)^2}\left\{-\frac{\alpha_1}{4}\sigma^2\right\}<0\,.
\end{eqnarray*}

\medskip

Now consider  the region
$$\Omega^{\varepsilon-\varepsilon_1}\setminus \overline{\Omega^{\varepsilon}}=\{x\in \Omega \;|\; \;\varepsilon-\varepsilon_1< d(x) \leq \varepsilon\}\,$$
For any  $x\in \Omega^{\varepsilon-\varepsilon_1}\setminus\overline{\Omega^{\varepsilon}}, t\in (\tau-\delta, \tau)$, thanks to \eqref{estimate:1}, \eqref{laplace-d}, \eqref{lb-gradd}, \eqref{condition-beta}, \eqref{eqagg20} we have:
\begin{eqnarray*}
&& \partial_t\xi+\frac 52 a\,|\nabla\xi|^2+{\rm{div}}(a\,\nabla \xi)\\
&&\qquad\leq \frac1{2(s-\alpha_1 t)^2}\Big\{\zeta^2(-\alpha_1+5c_0\beta^2) \\&&\qquad\quad+2(s-\alpha_1t)\beta \big(\zeta [c_1 d^{\gamma-\beta-2}+ c_0k_0 d^{\gamma-\beta-1}]-\beta\tilde c_0 \nu_0\big)\Big\}\,.
\end{eqnarray*}
Observe that for any $x\in \Omega^{\varepsilon-\varepsilon_1}\setminus\overline{\Omega^{\varepsilon}}$,
\begin{equation}\label{eqagg22}
\zeta(x)\leq (\varepsilon-\varepsilon_1)^{-\beta}-\varepsilon^{-\beta}=\varepsilon^{-\beta}[(1-c)^{-\beta}-1],
\end{equation}
while
\begin{equation}\label{eqagg23}
d(x)<\varepsilon\,.
\end{equation}
In view of \eqref{eqagg32}, \eqref{eqagg33}, \eqref{eqagg22}, \eqref{eqagg23}, we obtain
\begin{eqnarray*}
&&\partial_t\xi+\frac 52 a(x,t)|\nabla\xi|^2+{\rm{div}}(a(x,t)\nabla \xi)\\
 &&\qquad\leq \frac{\beta}{s-\alpha_1 t}\left\{[(1-c)^{-\beta}-1](c_1+c_0k_0)-\beta\tilde c_0\nu_0 \right\}<0,
\end{eqnarray*}

thanks to \eqref{eqagg21}.

\bigskip

{\textbf{Claim 2:}} Condition (E2) holds.

{\it Proof for Claim 2.} Using \eqref{eta-0} and \eqref{eta-1}  we have
\begin{eqnarray*}
 \eta\; \div (a(x,t)\nabla\eta)+\frac 52|\nabla \eta|^2a(x,t)
 &=&\eta\div a(x,t)\nabla \eta+\eta a(x,t) \Delta \eta+ \frac 52|\nabla \eta|^2a(x,t)\\
 &\leq&|\nabla a(x,t)||\nabla \eta|+ |a(x,t)| |\Delta \eta|+ \frac 52|\nabla \eta|^2|a(x,t)|\\
 &\leq&c_0\gamma d^{\gamma-1}\frac{A_1}{\varepsilon}+c_0d^{\gamma}\frac{1}{\varepsilon^{2}}+\frac 52\frac{A_1^2}{\varepsilon^2}c_0d^{\gamma}\\
 &=&\frac{1}{\varepsilon^2}d^{\gamma}\left(d^{-1}A_1 c_1 \varepsilon+c_0+\frac 52 A_1^2 c_0\right)\,.
\end{eqnarray*}
%
Because of the support conditions of $\nabla\eta$ and $\Delta \eta$ (contained in the set $\Omeet\setminus \Omega^{\frac 23 \varepsilon}=\{x\in \Omega: \frac{\varepsilon}{2}<d\leq \frac 23\varepsilon\}$), the term $d^{-1}A_1 c_1 \varepsilon+c_0+\frac 52 A_1^2 c_0$ is bounded by a constant independent of $\varepsilon$, and the claim follows.

\bigskip

 Finally, inserting (E1) and (E2) in \eqref{eq:4} we obtain
 \begin{equation*}
 \begin{array}{rlll}
 \int_{\Omeet}\psia(u(x,\tau))\eta^2(x)e^{\xi(x,\tau)}\, dx
 &\leq &\int_{\Omeet}\psia(u(x,\tau-\delta))\eta^2(x)e^{\xi(x,\tau-\delta)}\, dx &\\
 \\
 &&\quad+ C_1 \varepsilon^{\gamma-2}\iint_{\tau-\delta}^{\tau}\int_{\Omega^{\frac{\varepsilon}{2}}\setminus\Omega^{\frac{2}{3}\varepsilon}}\psia(u(x,t)) e^{\xi(x,t)} \, dx\, dt\,,
 \end{array}
\end{equation*}
with $C_1>0$, independent of $\varepsilon$, as in $(E_2)$. This completes the proof.
\end{proof}

\bigskip

\begin{proof}[Proof of Proposition \ref{lem1}]
For every $\alpha>0$ define
$$\psia(z)=(z^2+\alpha)^{\frac 12}, \qquad z\in \R\,,$$
with $\alpha>0$. Since $\psia''\geq 0$, in view of \eqref{eqagg40} we can infer that
$\psia$ is a subsolution of \eqref{TD}. The application of Lemma \ref{pr1} yields
 \begin{equation}\label{start}
\begin{array}{rlll}
\int_{\Omeet}\psia(u(x,\tau))\eta^2(x)e^{\xi(x,\tau)}\, dx
&\leq&\int_{\Omeet}\psia(u(x,\tau-\delta))\eta^2(x)e^{\xi(x,\tau-\delta)}\, dx &\\
\\
&+& C_1\varepsilon^{\gamma-2} \iint_{\Omeet\setminus \Omega^{\frac 23 \varepsilon}\times (\tau-\delta, \tau)}\psia(u(x,t)) e^{\xi(x,t)} \, dx\, dt\,;
\end{array}
\end{equation}
here $\xi$ is defined as in \eqref{xi}, and conditions \eqref{condition-beta} and \eqref{eqagg35} hold.

 Using that $\Omeps\subset \Omeet$, $\eta=1$ on $\Omega ^{\frac 23\varepsilon}$
and the positivity of the integrand we have
 $$\int_{\Omeps}\psia(u(x,\tau))\eta^2(x)e^{\xi(x,\tau)}\, dx=\int_{\Omeps}\psia(u(x,\tau))e^{\xi(x,\tau)}\, dx\leq \int_{\Omeet}\psia(u(x,\tau))e^{\xi(x,\tau)}\, dx\,. $$
 Letting  $\alpha\rightarrow 0^+$ in \eqref{start}, applying the Lebesgue's
 dominated convergence theorem and observing that $0\leq\eta\leq 1$ we obtain
  \begin{equation}\label{L-dc}
\begin{array}{rlll}
\int_{\Omeps}|u(x,\tau)|e^{\xi(x,\tau)}\, dx
&\leq&\int_{\Omeet}|u(x,\tau-\delta)|e^{\xi(x,\tau-\delta)}\, dx &\\
\\
&+& C_1\varepsilon^{\gamma-2} \iint_{\Omeet\setminus \Omega^{\frac 23 \varepsilon}\times (\tau-\delta, \tau)}|u(x,t)| e^{\xi(x,t)} \, dx\, dt\,.
\end{array}
\end{equation}
Recalling \eqref{zeta}, we first notice that $\xi=0$ in $\Omeps$ for any
$t\in [\tau-\delta, \tau]$. Choose $s$ as in \eqref{eqagg35}. Therefore,  $\xi(x,t)<0$ for all $x\in \Omega \setminus \overline{\Omeps}$
and $t\in (\tau-\delta,\tau)$.
Since $$\Omeet\setminus \Omega^{\frac 23\varepsilon}=\{x\in \Omega: \frac{\varepsilon}{2}<d(x)\leq \frac{2}{3}\varepsilon\}\subset \Omega \setminus \Omeps,$$ then
$$\zeta(x)=d^{-\beta}-\varepsilon^{-\beta}= \varepsilon^{-\beta}\left(\left(\frac{\varepsilon}{d}\right)^{\beta}-1\right) \quad \mbox{ in } \Omeet\setminus \Omega^{\frac 23\varepsilon}\,.$$
Using
$$\left(\frac{\varepsilon}{d}\right)^{\beta}\geq \left(\frac 32\right)^{\beta} \quad \mbox{ for }\beta>0\,,$$
we have
$$\zeta(x)\geq \tilde{C}\varepsilon^{-\beta} \qquad \mbox{ with } \tilde{C}=\left(\frac 32\right)^{\beta}-1\,.$$
Moreover, we have
$$\xi(x,t)=-\frac{\zeta^2(x)}{2(s-\alpha_1 t)}\leq -\frac{\tilde{C}^2 \varepsilon^{-2\beta}}{4\alpha_1 \delta} \qquad \mbox{ in } \Omeet\setminus \Omega^{\frac 23\varepsilon} \,.$$
Finally, inserting this bound in \eqref{L-dc} we have
  \begin{equation*}
\begin{array}{rlll}
\int_{\Omeps}|u(x,\tau)|\, dx
&\leq&\int_{\Omeet}|u(x,\tau-\delta)|\, dx &\\
\\
&+&  C_1\varepsilon^{\gamma-2}e^{-\frac{\tilde{C}^2 \varepsilon^{-2\beta}}{4\alpha_1 \delta} } \iint_{\Omeet\setminus \Omega^{\frac 23 \varepsilon}\times (\tau-\delta, \tau)}|u(x,t)| \, dx\, dt\,.
\end{array}
\end{equation*}
Due to condition \eqref{eq4af}, it follows that
\begin{equation*}
\begin{array}{rlll}
\int_{\Omeps}|u(x,\tau)|\, dx
&\leq&\int_{\Omeet}|u(x,\tau-\delta)|\, dx &\\
\\
&+& C_1 \operatorname{meas}(\Omega) T\varepsilon^{\gamma-2}e^{-\frac{\tilde{C}^2 \varepsilon^{-2\beta}}{4\alpha_1 \delta} } e^{\theta \varepsilon^{(2-\gamma)}}\,.
\end{array}
\end{equation*}
In view of \eqref{condition-beta}, since $0<\delta\leq\frac{\tilde{C}^2}{4\theta\alpha_1}$, we obtain
\begin{equation*}
\begin{array}{rlll}
\int_{\Omeps}|u(x,\tau)|\, dx
&\leq&\int_{\Omeet}|u(x,\tau-\delta)|\, dx + C_1 \operatorname{meas}(\Omega) T\varepsilon^{\gamma-2}\,.
\end{array}
\end{equation*}
This completes the proof.
\end{proof}

\begin{proof}[Proof Lemma \ref{lemma2}]
The thesis follows by arguing as in the proof of  \cite[Proposition 4.1]{Pu2}. The only difference is that instead of \eqref{e3} in \cite[Proposition 4.1]{Pu2} there was
\begin{equation}\label{e3cbis}
\int_{\Omega^{\e}} u^2(x, \tau)\,dx \leq \int_{\Omega^{\frac{\epsilon}2}} u^2(x, \tau-\delta)\, dx + \hat{C}\e^{\mu}\,.
\end{equation}
However, such difference dose not affect the argument to get the conclusion.
\end{proof}

\section{Proof of Theorem \ref{th2}}\label{sth2}
Let $\beta=-\gamma+2$ whenever $\gamma\in (1,2);$ let $\beta=b>0$ whenever $\gamma=2$, with $b>0$ arbitrary.
Consider $\ell\in \left(0,\frac 12\right)$ such that
\begin{equation}\label{eqagg30}
[1-(1-\ell)^{\beta}](c_1+c_0k_0)-\beta\nu_0\tilde c_0 <0\,,
\end{equation}
and define
\begin{equation}\label{eqagg31}
\bar \sigma:=1-(1-\ell)^{\beta}\,.
\end{equation}

Define the functions
\begin{equation}\label{zeta2}
 \zeta(x)=\begin{cases}
            0 &  \mbox{ for } x\in \Omeps\,\\
            \varepsilon^{\beta}-d^{\beta}(x) & \mbox{ for } x\in \Omega\setminus \Omeps\,,
           \end{cases}
\end{equation}
and
\begin{equation}\label{eqagg34}
\xi(x)=-\frac{\zeta^2(x)}{2(s-\alpha_1 t)}\quad \textrm{for all }\,\, x\in \Omega, t\neq \frac s{\alpha_1}\,.
\end{equation}

\smallskip

The proof of Theorem \ref{th2} is based on the combination of the following results.

\begin{prop}\label{lem3}
  Under assumption \eqref{a3} with $\gamma\in [1,2]$, suppose $u\in C^{2,1}(Q_T)\cap C(\Omega\times [0,T])$ solves \eqref{TD2}. Let $\tau>0, b>0,$ $\ell\in \left(0, \frac 12 \right)$ be such that \eqref{eqagg30} is true, $\bar \sigma$ be defined by \eqref{eqagg31}. Suppose that
  \begin{equation*}
      0<\delta<
      \begin{cases}\min\left\{
      \frac{\bar \sigma^2}{16(-\gamma+2)(c_1 +c_0)}\varepsilon^{-2\gamma+4}, \tau\right\} & \textrm{ for }\gamma\in[1,2) \\
      \min\left\{\frac{\bar\sigma^2}{16b[c_1 +(b-1)^++c_0]}, \tau\right\} & \textrm{ for }\gamma=2,
      \end{cases}
  \end{equation*}
  \begin{equation*}
  \alpha_1 \geq\begin{cases}\max\left\{ \frac{40c_0(-\gamma+2)^2}{\bar\sigma^2}, 5c_0(2-\gamma)^2\right\} &  \textrm{ for }\gamma\in [1,2)\\
  \max\left\{\frac{40c_0b^2}{\bar\sigma^2}, 5c_0 b^2\right\} & \textrm{ for }\gamma=2\,,
  \end{cases}
\end{equation*}
  and that, for some $C>0, \mu>0$,
 \begin{equation}\label{ub:exp-growth2}
  \int_0^T\int_{\Omeet\setminus \Omega^{\frac{2}{3}\varepsilon}}|u(x,t)|d(x)^{\gamma-2}\,dx\, dt\leq C\varepsilon^{\mu} \mbox{ for every } \varepsilon\in (0,\varepsilon_0)\,.
 \end{equation}
Then
\begin{equation}\label{est-lem}
\begin{array}{rlll}
\int_{\Omeps}|u(x,\tau)|\, dx
&\leq&\int_{\Omeet}|u(x,\tau-\delta)|\, dx + \tilde{C}\varepsilon^{\mu}\,,
\end{array}
\end{equation}
for some constant $\tilde C>0$ independent of $\varepsilon$.
\end{prop}

\begin{lemma}\label{lemma2bis}
Let $u\in C(\Omega\times [0, T])$ with
\begin{equation}\label{e88}
u=0\quad \textrm{in}\;\; \Omega\times\{0\}\,.
\end{equation}
Suppose that there exist $c>0, \hat C>0, \e_0>0, \mu_2>\mu_1>0$ such that for any $\e\in (0, \e_0)$, $\tau\in (0,T)$,
\begin{equation}\label{e22}
0<\delta< \min\{\tau, c \e^{\mu_1}\}\,,
\end{equation}
there holds
\begin{equation}\label{e33}
\int_{\Omega^{\e}} u(x, \tau)\,\dd x \leq \int_{\Omega^{\frac{\epsilon}2}} u(x, \tau-\delta) \dd x + \hat{C} \e^{\mu_2}\,.
\end{equation}
Then
\[u\equiv 0\quad \textrm{in}\;\; \Omega\times (0, T]\,.\]
\end{lemma}
Lemma \ref{lemma2bis} is an extension of Lemma \ref{lemma2}. Differently from Lemma \ref{lemma2}, in Lemma \ref{lemma2bis} the bound on $\delta$ goes to zero as $\varepsilon\to 0^+$. To manage this situation, the condition $\mu_2>\mu_1$ will be expedient.

\begin{proof}[Proof of Theorem \ref{th2}] The thesis follows,
  combining Proposition \ref{lem3} and Lemma \ref{lemma2bis}, with $$\mu_1=-2\gamma+4, \quad \mu_2=\mu>\mu_1,$$
 $$
 c=\begin{cases}
 \frac{\bar{\sigma}^2}{16(-\gamma+2)(c_1+c_0)} & \mbox{ for } \gamma\in [1,2)\\
 \frac{\bar{\sigma}^2}{16b[c_1+(b-1)^++c_0]} & \mbox{ for } \gamma=2\,,
 \end{cases}
 $$
 and
 $$\tilde{C}=\hat{C}\,.$$
\end{proof}

\subsection{Proofs of Proposition \ref{lem3} and Lemma \ref{lemma2bis}}

The proof of Proposition \ref{lem3} is based on the following crucial lemma.

  \begin{lemma}\label{prop2}
    Under assumption \eqref{a3} with $\gamma\in [1,2]$, suppose $u\in C^{2,1}(Q_T)\cap C(\Omega\times [0,T])$ solves \eqref{TD2}. Let $\varepsilon\in (0, \varepsilon_0), \tau>0, b>0,$ $\ell\in \left(0, \frac 12 \right)$ be such that \eqref{eqagg30} is true, $\bar \sigma$ be defined by \eqref{eqagg31}. If
    \begin{equation}\label{eqagg32}
        0<\delta<
        \begin{cases}\min\left\{
        \frac{\bar \sigma^2}{16(-\gamma+2)(c_1 +c_0)}\varepsilon^{-2\gamma+4}, \tau\right\} & \textrm{ for }\gamma\in[1,2) \\
        \min\left\{\frac{\bar\sigma^2}{16b[c_1 +(b-1)^++c_0]}, \tau\right\} & \textrm{ for }\gamma=2,
        \end{cases}
    \end{equation}
    \begin{equation}\label{eqagg33}
    \alpha_1 \geq\begin{cases}\max\left\{ \frac{40c_0(-\gamma+2)^2}{\sigma^2}, 5c_0(2-\gamma)^2\right\} &  \textrm{ for }\gamma\in [1,2)\\
    \max\left\{\frac{40c_0b^2}{\bar\sigma^2}, 5c_0 b^2\right\} & \textrm{ for }\gamma=2\,,
    \end{cases}
    \end{equation}

    then
  \begin{equation}\label{bound-prop2}
    \begin{array}{rlll}
    \int_{\Omeet}\psia(u(x,\tau))\eta^2(x)e^{\xi(x,\tau)}\, dx
    &\leq&\int_{\Omeet}\psia(u(x,\tau-\delta))\eta^2(x)e^{\xi(x,\tau-\delta)}\, dx &\\
    \\
    &+& C_1\int_{\tau-\delta}^{\tau}\int_{\Omega^{\frac{\varepsilon}{2}}\setminus\Omega^{\frac{2}{3}\varepsilon}}\psia(u(x,t)) e^{\xi(x,t)} d^{\gamma-2}(x)\, dx\, dt\,,
    \end{array}
   \end{equation}
   where $\xi$ is defined as in \eqref{eqagg34} with $s=\alpha_1(\tau+\delta)$, for a suitable $C_1>0$ independent of $\varepsilon$.
  \end{lemma}

\begin{proof}[Proof of Lemma \ref{prop2}] Let $\zeta$ be defined by \eqref{zeta2}.
Then
\begin{equation}
 \nabla \zeta=-\beta d^{\beta-1}\nabla d \qquad \mbox{ and } \qquad \Delta \zeta=-\beta(\beta-1)d^{\beta-2}(\nabla d)^2-\beta d^{\beta-1}\Delta d\,.
\end{equation}
Let $\xi$ be defined as in \eqref{eqagg34}. Imitating the arguments in Proposition \ref{pr1}, we derive
\begin{equation}\label{eq:3}
\begin{array}{rlll}
&&\int_{\Omeet}\psia(u(x,\tau))\eta^2(x)e^{\xi(x,\tau)}\, dx\\
\\
&&\leq\int_{\Omeet}\psia(u(x,\tau-\delta))\eta^2(x)e^{\xi(x,\tau-\delta)}\, dx &\\
\\
&&\quad+ \iint_{\C}\psia(u(x,t)) e^{\xi(x,t)} [\eta(x){\rm{div}} (a(x,t)\nabla\eta(x))+\frac 52|\nabla \eta(x)|^2a(x,t)]\, dx\, dt\\
\\
&&\quad+\iint_{\C}\psia(u(x,t)) e^{\xi(x,t)}\eta^2(x)[\partial_t\xi+\frac 52|\nabla \xi(x,t)|^2a(x,t)+{\rm{div}} (a(x,t)\nabla \xi(x,t))]\, dx\, dt\,,
\end{array}
\end{equation}
which is exactly \eqref{eq:4}.
Our next goal is to ensure that the following two conditions
\begin{align*}
&\partial_t\xi+\frac 52 a\,|\nabla\xi|^2+{\rm{div}}(a\,\nabla \xi)\leq 0 \quad \mbox{ in }[\Omega\setminus \partial\Omeps]\times(\tau-\delta,\tau)\, ; \tag{D1}\\
\\
 &\eta\, {\rm{div}} (a\,\nabla\eta)+\frac 52|\nabla \eta|^2a\,\leq C_1d^{\gamma-2} \quad  \mbox{ in } \Omega\times(\tau-\delta,\tau) \tag{D2}
 \end{align*}
are simultaneously satisfied.

\smallskip

{\textbf{Claim 3:}} Condition (D1) holds.

{\it Proof of for Claim 3}.
By the same arguments used to obtain \eqref{estimate:1}, we deduce that
\begin{equation}\label{estimate:2.1}
\begin{array}{rlll}
 &&\partial_t\xi+\frac 52 a(x,t)|\nabla\xi|^2+{\rm{div}}(a(x,t)\nabla \xi)\\
 \\
 &&=\frac{1}{2(s-\alpha_1 t)^2}\left\{- \alpha_1\zeta^2(x)+5 \beta^2 a(x,t)\,\zeta^2(x)\,d^{2\beta-2}(x)\,|\nabla d(x)|^2\right.\\
 \\
 &&\left.\quad +2(s-\alpha_1 t)\beta\left[\zeta(x)\nabla a(x,t)\cdot \nabla d(x) d^{\beta-1}(x)-\beta a(x,t) d^{2\beta-2}(x)|\nabla d(x)|^2\right.\right.\\
\\
 &&\left.\left.\quad+(\beta-1)a(x,t)\zeta(x) d^{\beta-2}(x)|\nabla d(x)|^2+a(x,t)\zeta(x) d^{\beta-1}(x)\Delta d(x)\right]\right\}.
 \end{array}
 \end{equation}
 Also here, because of \eqref{lb-gradd} and the non-negativity of $a$, we have
 $$-\beta a(x,t) d^{2\beta-2}(x)|\nabla d(x)|^2\leq 0\,.$$
 We now analyze all the other terms on the right-hand-side singularly, using the
fact that in $\Omega \setminus \Omeps$ we have $d(x)\leq\varepsilon$. We start
with the second term:
 \begin{equation*}
  5\beta^2 a(x,t)\zeta^2(x)\,d^{2\beta-2}(x)\,|\nabla d(x)|^2\leq  10 \beta^2 c_0\varepsilon^{2\beta} d^{2\beta-2+\gamma}(x) \leq 10\beta^2c_0 \varepsilon^{4\beta+\gamma-2}\,,
 \end{equation*}
 where we used \eqref{a3}, \eqref{grad-d} and that $\zeta^2\leq 2\varepsilon^{2\beta}$,
 together with the hypothesis that $\gamma+2\beta-2\geq 0$ for the last inequality.
 For the third term we use again \eqref{a3} and \eqref{grad-d} to obtain
 $$\zeta(x)\nabla a(x,t)\cdot \nabla d(x) d^{\beta-1}(x)\leq \varepsilon^{\beta}c_1d^{\gamma-1}(x)|\nabla d(x)|d^{\beta-1}(x)\leq c_1 \varepsilon^{\beta+\gamma-2}\,,$$
 where the last inequality  holds if $\gamma+\beta-2\geq0$.
 Again, if $\gamma+\beta-2\geq0$, the fourth term is estimated easily as
 $$(\beta-1)a(x,t)\zeta(x)d^{\beta-2}(x)|\nabla d(x)|^2\leq c_0(\beta-1)\varepsilon^{\beta}d^{\beta+\gamma-2}(x)\leq(\beta-1)c_0\varepsilon^{2\beta+\gamma-2}\,.$$
 Finally, we estimate the last term as
 $$a(x,t)\zeta(x) d^{\beta-1}(x)\Delta d(x)\leq c_0 \varepsilon^{\beta}d^{\beta+\gamma-1}(x)k_0\leq c_0 \varepsilon^{2\beta+\gamma-1} $$
if $\gamma+\beta-1\geq 0$.
Collecting these estimates in the range $1\leq\gamma\leq2$ we choose $\beta\geq-\gamma+2$,
so that all the previous conditions are satisfied. In particular we set
\begin{equation*}
  \beta=
  \begin{cases}
     -\gamma+2 & \mbox{ for } 1\leq\gamma<2\\
     b & \mbox{ for } \gamma=2
  \end{cases}
\end{equation*}
where $b$ is any positive number.

Finally choose
\begin{equation}\label{eqagg355}
s=\alpha_1(\tau +\delta)\,.
\end{equation}
so, for all $t\in (\tau-\delta, \tau),$
\begin{equation}\label{eqagg200}
\alpha_1\delta<s-\alpha_1t<2\alpha_1\delta\,.
\end{equation}

We now write the set $\Omega\setminus \overline{\Omeps}$ as a union of two disjoint sets
$$\Omega\setminus \overline{\Omeps}=(\Omega\setminus \Omega^{\varepsilon-\varepsilon_2})\cup (\Omega^{\varepsilon-\varepsilon_2}\setminus \overline{\Omega^{\varepsilon}})\,,$$
and analyze the validity of condition $(D1)$ separately in the two domains.
First let us consider the set
$\Omega\setminus \Omega^{\varepsilon-\varepsilon_2}=\{x\in \Omega: d(x)\leq \varepsilon-\varepsilon_2\}$
and look at the case $\gamma\in[1,2)$ and $\gamma=2$ separetely.
\begin{itemize}
  \item
For the case $\gamma\in [1,2)$ and $\beta=-\gamma+2$ we have
\begin{equation*}
\begin{array}{rlll}
 & & \partial_t \xi + \frac 52 a(x,t) |\nabla\xi|^2 + {\rm{div}} (a(x,t) \nabla \xi ) &\\
 \\
 & & \leq \frac{1}{2(s-\alpha_1 t)^2} \left\{ - \alpha_1\zeta^2(x) +  10 c_0\beta^2 \varepsilon^{-3\gamma+6} \right. &\\
 \\
 & &\qquad + \left. 2 (s-\alpha_1t) (-\gamma+2) \left[ c_1 +c_0(-\gamma+1)\varepsilon^{-\gamma+2}+c_0\varepsilon^{-\gamma+3} \right] \right\} \,.&
\end{array}
\end{equation*}
 For the first term on the right-hand-side we claim the existence of a $\bar\sigma\in (0,1]$ such that
 $$-\zeta^2\leq -\bar{\sigma}^2 \varepsilon^{2\beta}\,$$
and this is equivalent to the condition
\begin{equation}\label{this-cond}
  \bar{\sigma}\leq 1-\left(\frac{d}{\varepsilon}\right)^{\beta}\,.
\end{equation}
We set $\varepsilon_2=\ell\varepsilon$ with $\ell\in (0,\frac 12)$ and choose $\bar{\sigma}=1-(1-\ell)^{\beta}$
 so that \eqref{this-cond} is trivially satisfied.
Thus
\begin{equation*}
\begin{array}{rlll}
&&\partial_t\xi+\frac 52 a(x,t)|\nabla\xi|^2+{\rm{div}}(a(x,t)\nabla \xi)\\
\\
&&\leq\frac{1}{2(s-\alpha_1 t)^2}\left\{- \alpha_1\bar{\sigma}^2 \varepsilon^{-2\gamma+4}+   10 c_0(-\gamma+2)^2 \varepsilon^{-3\gamma+6}
\right.\\
\\
&& \qquad\left.+2(s-\alpha_1t)(-\gamma+2) \left[ c_1 +c_0(-\gamma+1)\varepsilon^{-\gamma+2}+c_0\varepsilon^{-\gamma+3} \right]\right\}\,,
\end{array}
\end{equation*}
 and, using that $0<s-\alpha_1 t<2\alpha_1\delta$, we obtain
 \begin{equation*}
 \begin{array}{rlll}
  &&\partial_t\xi+\frac 52 a(x,t)|\nabla\xi|^2+{\rm{div}}(a(x,t)\nabla \xi)\\
  \\
  &&\leq\frac{1}{2(s-\alpha_1 t)^2}\left \{- \alpha_1\bar{\sigma}^2 \varepsilon^{-2\gamma+4}+ 10 c_0(-\gamma+2)^2 \varepsilon^{-3\gamma+6}\right.\\
  \\
  &&\qquad\left.+4\alpha_1\delta(-\gamma+2) \left[ c_1+c_0(-\gamma+1)\varepsilon^{-\gamma+2}+c_0\varepsilon^{-\gamma+3} \right]\right\}\\
  \\
  &&\leq\frac{\varepsilon^{-2\gamma+4}}{2(s-\alpha_1t)^2}\left\{-\alpha_1\bar{\sigma}^2+10c_0(-\gamma+2)^2\varepsilon^{-\gamma+2}\right.\\
  \\
  &&\qquad\left.+4\alpha_1\delta (-\gamma+2)\varepsilon^{2\gamma-4}[c_1 +c_0\varepsilon^{-\gamma+3}]\right\}\,,
  \end{array}
 \end{equation*}
 where in the last inequality we used $(-\gamma+1)\varepsilon^{-\gamma+2}\leq 0$ since $\gamma\geq1$.
  Comparing the three terms (the first with the third and then the first with the second) we obtain
  \begin{equation*}
   \partial_t\xi+\frac 52 a(x,t)|\nabla\xi|^2+{\rm{div}}(a(x,t)\nabla \xi)\leq\frac{\varepsilon^{-3\gamma+6}}{2(s-\alpha_1 t)^2}\left\{- \frac{\alpha_1}{4}\bar{\sigma}^2 \right\}<0
   \,
  \end{equation*}
  if the following two conditions are satisfied:
  \begin{equation*}
    0<\delta\leq \frac{\bar{\sigma}^2\varepsilon^{-2\gamma+4}}{16(-\gamma+2)[c_1 +c_0\varepsilon^{-\gamma+3}]}\qquad \mbox{ and } \qquad\alpha_1\geq \frac{40c_0(-\gamma+2)^2\varepsilon^{-\gamma+2}}{\bar{\sigma}^2}\,.
  \end{equation*}

 \item For the case $\gamma=2$ and $\beta=b>0$ we have
 \begin{equation*}
 \begin{array}{rlll}
  & & \partial_t \xi + \frac 52 a(x,t) |\nabla\xi|^2 + {\rm{div}} (a(x,t) \nabla \xi ) &\\
  \\
  & &\leq \frac{1}{2(s-\alpha_1 t)^2} \left\{ - \alpha_1\zeta^2(x) +  10 c_0b^2 \varepsilon \right. &\\
  \\
  & & \qquad+ \left. 2 (s-\alpha_1t) \beta \left[ c_1 +(b-1)^++c_0\varepsilon \right] \right\} \,.
 \end{array}
\end{equation*}
 Proceeding as above, we deduce easily
 \begin{equation*}
  \partial_t\xi+\frac 52 a(x,t)|\nabla\xi|^2+{\rm{div}}(a(x,t)\nabla \xi)\leq\frac{1}{2(s-\alpha_1 t)^2}\left\{- \frac{\alpha_1}{4}\bar{\sigma}^2 \right\}<0
  \,
 \end{equation*}
 if the following two conditions are satisfied:
 \begin{equation*}
   0<\delta\leq \frac{\bar{\sigma}^2}{16b[c_1 +(b-1)^++c_0\varepsilon]}\qquad \mbox{ and } \qquad\alpha_1\geq \frac{40c_0b^2}{\bar{\sigma}^2}\,.
 \end{equation*}
 \end{itemize}

\medskip

Now consider the region
$$\Omega^{\varepsilon-\varepsilon_1}\setminus \overline{\Omega^{\varepsilon}}=\{x\in \Omega \;|\; \;\varepsilon-\varepsilon_1< d(x)\leq \varepsilon\}\,$$
For any  $x\in \Omega^{\varepsilon-\varepsilon_1}\setminus\Omega^{\varepsilon}, t\in (\tau-\delta, \tau)$, thanks to \eqref{estimate:2.1}, \eqref{laplace-d}, \eqref{lb-gradd},  \eqref{eqagg200} we have:
\begin{eqnarray*}
&& \partial_t\xi+\frac 52 a(x,t)|\nabla\xi|^2+{\rm{div}}(a(x,t)\nabla \xi)\\
&&=\frac1{2(s-\alpha_1 t)^2}\Big\{\zeta^2(-\alpha_1+5c_0\beta^2)\Big.\\
&&\Big.\qquad+2(s-\alpha_1t)\beta\Big[ \zeta \left(c_1 d^{\gamma+\beta-2}(x)+ c_0k_0 d^{\gamma+\beta-1}(x)\right)
-d^{2\beta-2+\gamma}(x)\beta\tilde c_0 \nu_0\Big]\Big\}\,.
\end{eqnarray*}
Observe that for any $x\in \Omega^{\varepsilon-\varepsilon_1}\setminus\overline{\Omega^{\varepsilon}}$,
\begin{equation}\label{eqagg222}
\zeta(x)\leq\varepsilon^{-\beta}- (\varepsilon-\varepsilon_1)^{-\beta}=\varepsilon^{-\beta}[1-(1-\ell)^{\beta}],
\end{equation}
while
\begin{equation}\label{eqagg233}
d(x)<\varepsilon\,.
\end{equation}
In view of \eqref{eqagg2}, \eqref{estimate:1}, \eqref{eqagg222}, \eqref{eqagg233}, we obtain
\begin{eqnarray*}
&&\partial_t\xi+\frac 52 a(x,t)|\nabla\xi|^2+{\rm{div}}(a(x,t)\nabla \xi)\\
&&\qquad\leq \frac{\beta\varepsilon^{2\beta-2+\gamma}}{s-\alpha_1 t}\left\{[(1-\ell)^{\beta}-1](c_1+c_0k_0)-\beta\nu_0\tilde c_0 \right\}<0,
\end{eqnarray*}
thanks to \eqref{eqagg30}.

\bigskip

 {\textbf{Claim 4:}} Condition (D2) holds.

 {\it Proof of Claim 4.} Using the properties of $\eta$ in \eqref{eta-0} and \eqref{eta-1} and the assumption on $a$ in \eqref{a3}, we have
 \begin{eqnarray*}
   \eta{\rm{div}} (a(x,t)\nabla\eta)+\frac 52|\nabla \eta|^2a(x,t)
   &\leq& c_1d^{\gamma-1}\frac{A_1}{\varepsilon}+c_0d^{\gamma}\frac{A_2}{\varepsilon^2}+\frac 52 \frac{A_1^2}{\varepsilon^2}c_0 d^{\gamma}\\
   &\leq& C_1  d^{\gamma-2}\,
 \end{eqnarray*}
 in $\Omega^{\frac{\varepsilon}{2}}\setminus\Omega^{\frac{2}{3}\varepsilon}$.

 Finally, inserting the estimates in (D1) and (D2) in \eqref{eq:3} we obtain
  \begin{equation*}
  \begin{array}{rlll}
  \int_{\Omeet}\psia(u(x,\tau))\eta^2(x)e^{\xi(x,\tau)}\, dx
  &=&\int_{\Omeet}\psia(u(x,\tau-\delta))\eta^2(x)e^{\xi(x,\tau-\delta)}\, dx &\\
  \\
  &+& C_1\int_{\tau-\delta}^{\tau}\int_{\Omega^{\frac{\varepsilon}{2}}\setminus\Omega^{\frac{2}{3}\varepsilon}}\psia(u(x,t)) e^{\xi(x,t)} d^{\gamma-2}(x)\, dx\, dt\,,
  \end{array}
 \end{equation*}
which coincides with \eqref{bound-prop2}.
\end{proof}

\begin{proof}[Proof of Proposition \ref{lem3}]
  Using the same arguments as in Proposition \ref{lem1} and the Lebesgue's dominated convergence theorem,
   from \eqref{bound-prop2}, we have
  \begin{equation}\label{h}
  \begin{array}{rlll}
  \int_{\Omeps}|u(x,\tau)|e^{\xi(x,\tau)}\, dx
  &\leq&\int_{\Omeet}|u(x,\tau-\delta)|e^{\xi(x,\tau-\delta)}\, dx &\\
  \\
  &+& C_1\int_{\tau-\delta}^{\tau}\int_{\Omega^{\frac{\varepsilon}{2}}\setminus\Omega^{\frac{2}{3}\varepsilon}}|u(x,t)| e^{\xi(x,t)} d^{\gamma-2}(x)\, dx\, dt\,.
  \end{array}
  \end{equation}
  By the definition of \eqref{zeta2}, $\xi=0$ in $\Omeps$ for any $t\in [\tau-\delta, \tau]$. Choose $s$ as in \eqref{eqagg355}. So,  $\xi(x,t)<0$ for all $x\in \Omega\setminus\overline{\Omeps}$ and $t\in [\tau-\delta, \tau]$, so $e^{\xi(x,t)}  \leq 1$.
  %
Therefore, from \eqref{h} we obtain
   \begin{equation*}
   \begin{array}{rlll}
   \int_{\Omeps}|u(x,\tau)|\, dx
   &\leq&\int_{\Omeet}|u(x,\tau-\delta)|\, dx &\\
   \\
   &+& C_1 \iint_{\Omeet\setminus \Omega^{\frac 23 \varepsilon}\times (\tau-\delta, \tau)}|u(x,t)| d^{\gamma-2}(x) \, dx\, dt\,.
   \end{array}
   \end{equation*}
   Finally we use the assumption \eqref{ub:exp-growth2} to get \eqref{est-lem}.

\end{proof}


\begin{proof}[Proof of Lemma \ref{lemma2bis}]
Take any $\e>0, \tau\in (0, T)$. Define
\[\e_k:=\frac{\e}{k^{\frac 1{\mu_1}}}\quad \textrm{for all}\;\; k\in \mathbb N\,.\]
Note that
\begin{equation}\label{e6b}
\sum_{k=1}^{+\infty} \e_k^{\mu_1}=+\infty\,,
\end{equation}
while, since $\mu_2>\mu_1$,
\begin{equation}\label{e6bb}
\sum_{k=1}^{+\infty} \frac{1}{k^{\frac{\mu_2}{\mu_1}}}=:S<+\infty\,.
\end{equation}

Furthermore, let $\{\delta_k\}_{k\in \mathbb N}\subset [0, \infty)$ and $\{\tau_k\}_{k\in \mathbb N}\subset [0, \tau]$ be two sequences
with $\{\tau_k\}$ defined inductively as follows
\[ \tau_1:=\tau,\]
\[\tau_{k+1}:=\tau_k-\delta_k\quad \textrm{for every}\;\; k\in \mathbb N, k\geq 2\,, \]
and
\begin{equation}\label{e5}
0\leq \delta_k\leq  \min\{\tau_k, C_1\e_k^{\mu_1}\} \quad \textrm{for all}\;\; k\in \mathbb N\,.
\end{equation}
Observe that
\begin{equation}\label{e6}
\tau-\tau_{k+1}=\delta_1+\delta_2+\ldots + \delta_k \quad \textrm{for every}\;\; k\in \mathbb N\,.
\end{equation}

We can choose $\{\tau_k\}$ so that there exists $k_0\in \mathbb N$ with $\tau_{k_0+1}=0$. In fact, in view of \eqref{e6}, $\tau_{k_0+1}=0$ if and only if
\begin{equation}\label{e11}
\tau=\delta_0+\delta_1+\ldots+\delta_{k_0}\,.
\end{equation}
Due to \eqref{e6b}, we can select the sequence $\{\delta_k\}$, and thus $\{\tau_k\}$, so that \eqref{e5} and \eqref{e11} hold, for some $k_0\in \mathbb N$.

\smallskip

From \eqref{e33} it follows that for every $k=1\ldots k_0$
\begin{equation}\label{e7}
\int_{\Omega^{\e_k}} u(x, \tau_k)\, \dd x \leq \int_{\Omega_{\e_{k+1}}} u(x, \tau_{k+1}) \dd x + C_2 \e_k^{\mu_2}\,.
\end{equation}
Since $\tau_{k_0+1}=0$,  thanks to \eqref{e8} we get
\begin{equation}\label{e9}
\int_{\Omega^{\e_{k_0+1}}} u(x, \tau_{k_0+1})\, \dd x = \int_{\Omega^{\e_{k_0+1}}} u(x, 0)\, \dd x = 0\,.
\end{equation}
By iterating \eqref{e7} up to $k=k_0$, in view of \eqref{e9}, we get
\begin{equation}\label{e10}
\int_{\Omega^{\e}} u(x, \tau) \dd x \leq \int_{\Omega^{\e_{k_0+1}}} u(x, \tau_{k_0+1}) + C_2 \sum_{k=1}^{k_0} \e_k^{\mu_2}
= C_2 \sum_{k=1}^{k_0} \e_k^{\mu_2}\,.
\end{equation}
Thanks to \eqref{e6bb},
\[\sum_{k=1}^{k_0} \e_k^{\mu_2} =  \sum_{k=1}^{k_0} \frac{\e^{\mu_2}}{k^{\frac{\mu_2}{\mu_1}}} \leq \e^{\mu_2} \sum_{k=1}^{+\infty} \frac 1{k^{\frac{\mu_2}{\mu_1}}}= S \e^{\mu_2}
\mathop{\longrightarrow}_{\e\to 0}0\,.\]
Hence, by letting $\e\to 0^+$ in \eqref{e10}, we obtain
\[\int_{\Omega} u(x, \tau) \dd x = 0\,. \]
Since $\tau\in (0, T)$ was arbitrary, the conclusion follows.
\end{proof}

%

\end{document}